\documentclass{amsart}

\usepackage{graphicx}
\usepackage{hyperref}

\usepackage{subcaption}
\usepackage{tikz}
\usetikzlibrary{shapes}
\usetikzlibrary{positioning}

\theoremstyle{plain}
\newtheorem{theorem}{Theorem}
\newtheorem{corollary}[theorem]{Corollary}
\newtheorem{lemma}[theorem]{Lemma}

\newtheorem{problem}[theorem]{Problem}
\newtheorem{remark}[theorem]{Remark}

\begin{document}

\title{Bounds for the Vertex Chromatic Number of Connected Triangle-Free Graphs}

\author{S. Akbari$^{a}$, A. Beikmohammadi$^{b}$ \smallskip}

\maketitle

\begin{center}
$^a$Department of Mathematical Science, Sharif University of Technology \\
Email: \texttt{s\_akbari@sharif.edu} \\

\medskip
$^b$Department of Computer Science, Simon Fraser University \\
Email: \texttt{arash\_beikmohammadi@sfu.ca} \\
\end{center}

\begin{abstract}
It was recently shown that every connected graph of order $n \geq 5$ and size $m$ satisfies $\chi(G) \leq \left\lceil \frac{m}{\sqrt{n}} \right\rceil$, and it was asked whether the stronger inequality $\chi(G) \leq \left\lceil \frac{m}{2\sqrt{n}} \right\rceil + 1$ holds for every connected triangle-free graph. In this paper, we answer this question in the affirmative. In fact, we prove that every connected triangle-free graph $G$ with $G \not\cong C_5$ satisfies $\chi(G) \leq \left\lceil \frac{m}{\sqrt{5.5n}} \right\rceil + 1$, where the constant $\sqrt{5.5}$ cannot be replaced by any constant greater than or equal to $\sqrt{6}$, and the equality holds for the Gr\"{o}tzsch graph and for every odd cycle of length between $7$ and $21$. \\


\noindent MSC2020-Mathematics Subject Classification: 05C15, 05C35.
\end{abstract}

\section{Introduction}
Throughout this paper all graphs are simple, that is, with no loops and multiple edges. Let $G$ be a graph. We denote the edge set and the vertex set of $G$ by $E(G)$ and $V(G)$, respectively. The \textit{order} and \textit{size} of $G$ are $|V(G)|$ and $|E(G)|$, respectively. A graph $G$ is called an \textit{$(m,n)$-graph} if its size and its order are $m$ and $n$, respectively. As usual, $\delta(G)$ and $\Delta(G)$ denote the minimum degree and the maximum degree of $G$, and for simplicity we write $\delta$ and $\Delta$ instead of $\delta(G)$ and $\Delta(G)$. For a vertex $v$ of $G$, $N(v)$ denotes the set of neighbors of $v$ and $d(v) = |N(v)|$. We denote by $C_\ell$ the cycle of length $\ell$ and by $K_n$ the complete graph of order $n$. A graph is \textit{triangle-free} if it contains no $K_3$ as a subgraph, and $G - S$ denotes the graph obtained from $G$ by deleting the vertices of $S$ together with their incident edges.
The \textit{chromatic number} of $G$, denoted by $\chi(G)$, is the smallest number of colors needed to color all vertices of $G$ such that no pair of adjacent vertices gets the same color.
A graph $H$ is \textit{$k$-critical} if $\chi(H) = k$ and $\chi(H') < k$ for every proper subgraph $H'$ of $H$.
In~\cite{akbari2026some}, the following bound was established.

\begin{theorem}\cite{akbari2026some} \label{theorem: m over sqrt n}
Let $G$ be a connected $(m,n)$-graph of order at least $5$. Then
$$\chi(G) \leq \left\lceil \frac{m}{\sqrt{n}} \right\rceil.$$
\end{theorem}

The proof of Theorem~\ref{theorem: m over sqrt n} given in~\cite{akbari2026some} relies on an exhaustive computer search over all connected graphs of order at most $9$. The following problem was also proposed and confirmed by exhaustive computer search over all connected and triangle-free graphs of order at most $9$.

\begin{problem}\cite{akbari2026some} \label{problem: triangle-free}
Is it true that if $G$ is a connected and triangle-free graph of order $n$, then
$\chi(G) \le \left\lceil \frac{m}{2\sqrt{n}} \right\rceil + 1$?
\end{problem}

In this paper, we provide a proof for Problem~\ref{problem: triangle-free}.
Let $f(k)$ denote the minimum order of a triangle-free graph with chromatic number at least $k$, then $f(3) = 5$, $f(4) = 11$~\cite{chvatal1974minimality}, $f(5) = 22$~\cite{jensen1995small} and $32 \leq f(6) \leq 40$~\cite{goedgebeur2020minimal}. All our arguments rest on the elementary estimate $f(k) \geq \frac{k(k-1)}{2}+2$, which we prove in Section~\ref{section: preliminaries}, together with a lemma which converts a dense critical subgraph of a connected graph into a lower bound on the size of the whole graph.

Our main result, Theorem~\ref{theorem: triangle-free main} below, is the following strengthening of Problem~\ref{problem: triangle-free}: every connected triangle-free $(m,n)$-graph $G$ with $G \not\cong C_5$ satisfies
$$\chi(G) \leq \left\lceil \frac{m}{\sqrt{5.5n}} \right\rceil + 1,$$
and, since $\sqrt{5.5} > 2$, this settles Problem~\ref{problem: triangle-free} affirmatively. The constant $\sqrt{5.5}$ cannot be replaced by any constant greater than or equal to $\sqrt{6}$ (see Remark~\ref{remark: constant}). The same circle of ideas gives a proof of Theorem~\ref{theorem: m over sqrt n} which uses no computer search, confirms the observation of~\cite{akbari2026some} on the redundancy of the ceiling function in a sharp form.

\section{Preliminaries} \label{section: preliminaries}

\begin{lemma} \label{lemma: density}
Let $G$ be a connected $(m,n)$-graph with $\chi(G) = k \geq 3$ and let $H$ be a $k$-critical subgraph of $G$ of order $n_H$. Then
$$m \geq n + \frac{(k-3)\,n_H}{2}.$$
\end{lemma}

\begin{proof}
By~\cite[p.~194]{west2001introduction}, we have $\delta(H) \geq k-1$. So $m_H \geq \frac{(k-1)n_H}{2}$. Since $G$ and $H$ are connected, we have
$$m \geq m_H + n - n_H \geq \frac{(k-1)n_H}{2} + n - n_H = n + \frac{(k-3)n_H}{2},$$
as desired.
\end{proof}

The next lemma gives a lower bound for the order of a triangle-free graph with a given chromatic number. For an integer $k \geq 3$, let $f(k)$ denote the minimum order of a triangle-free graph whose chromatic number is at least $k$. Obviously, $f(3) = 5$.

\begin{lemma} \label{lemma: f(k)}
For every integer $k \geq 4$, $f(k) \geq f(k-1) + k - 1$. Consequently,
$$f(k) \geq \frac{k(k-1)}{2} + 2 = \frac{k^2-k+4}{2} \; , \qquad k \geq 3.$$
\end{lemma}

\begin{proof}
Let $G$ be a triangle-free graph with $\chi(G) \geq k$ and let $H$ be a $k$-critical subgraph of $G$. Clearly, $H$ is triangle-free and $\delta(H) \geq k-1$. Let $v \in V(H)$. Since $H$ is triangle-free, $N(v)$ is an independent set of $H$. We also have $|N(v)| = d(v) \geq k-1$. Deleting an independent set from a graph decreases its chromatic number by at most one, so $\chi(H - N(v)) \geq k-1$. Moreover $H - N(v)$ is triangle-free. Therefore $|V(H)| - |N(v)| \geq f(k-1)$, and we find that
$$|V(G)| \geq |V(H)| \geq f(k-1) + k - 1.$$
So $f(k) \geq f(k-1)+k-1$. Since $f(3) = 5$, by induction on $k$ we have
$$f(k) \geq 5 + \sum_{i=3}^{k-1} i = \frac{k(k-1)}{2} + 2,$$
as desired.
\end{proof}

\section{The bound $m/\sqrt{n}$ revisited}

It was observed in~\cite{akbari2026some} that the ceiling function in Theorem~\ref{theorem: m over sqrt n} appears to be redundant for all but finitely many connected graphs. The following result confirms this in a sharp form.

\begin{theorem} \label{theorem: no ceiling}
Let $G$ be a connected $(m,n)$-graph of order $n \geq 14$. Then
\[
  \chi(G) \leq \frac{m}{\sqrt{n}} .
\]
Moreover, the bound $14$ is best possible.
\end{theorem}

\begin{proof}
By contradiction suppose that $k = \chi(G) > \frac{m}{\sqrt{n}}$. Then
$$
m < k\sqrt{n} \leq n + \frac{k^2}{4}.
$$
If $k = 1$, then $n = 1$. If $k = 2$, then $n-1 \leq m < 2\sqrt{n}$, so $n \leq 5$.

Let $k \geq 3$ and let $H$ be a $k$-critical subgraph of $G$ of order $n_H \geq k$. By Lemma~\ref{lemma: density} and the fact that $m < n + \frac{k^2}{4}$,
$$
\frac{(k-3)n_H}{2} \leq m-n < \frac{k^2}{4}.
$$
Replacing $n_H$ by $k$ give us $2k(k-3) < k^2$, so $k \in \{3,4,5\}$. We consider these three cases separately, in each case combining the lower bound for $m$ given by Lemma~\ref{lemma: density} with the upper bound $m < k\sqrt{n}$.

If $k=3$, then $n \leq m < 3\sqrt{n}$, so $n \leq 8$.
If $k=4$, then $n_H \geq 4$ gives $m \geq n+2$, so $n + 2 < 4\sqrt{n}$, which implies $n \leq 11$.
If $k=5$, then $n_H \geq 5$ gives $m \geq n+5$, so $n+5 < 5\sqrt{n}$, which implies $n \leq 13$.
In every case $n \leq 13$, a contradiction.

For the sharpness, let $G$ be the graph obtained from $K_5$ by attaching $8$ pendant vertices to one of its vertices. Then $G$ is connected with $n = 13$, $m = 18$ and $\chi(G) = 5$, but
$$\frac{m}{\sqrt{n}} = \frac{18}{\sqrt{13}} \approx 4.993 < 5 = \chi(G).$$
Hence the conclusion of the theorem fails for this graph of order $13$.
\end{proof}

\section{Connected triangle-free graphs}

We are now ready to prove our main result.

\begin{theorem} \label{theorem: triangle-free main}
Let $G \not\cong C_5$ be a connected triangle-free $(m,n)$-graph. Then
$$
\chi(G) \leq \left\lceil \frac{m}{\sqrt{5.5n}} \right\rceil + 1.
$$
\end{theorem}

\begin{proof}
Clearly, the assertion holds for $n=1$. So assume that $n \geq 2$ and by contradiction suppose that $k = \chi(G) \geq \left\lceil \frac{m}{\sqrt{5.5n}} \right\rceil + 2$. Clearly, $\left\lceil \frac{m}{\sqrt{5.5n}} \right\rceil \geq 1$ and therefore $k \geq 3$. Moreover, $k - 2 \geq \frac{m}{\sqrt{5.5n}}$, so
$$
m \leq (k-2)\sqrt{5.5n} \leq n + \frac{5.5(k-2)^2}{4}.
$$
Let $H$ be a $k$-critical subgraph of $G$, of order $n_H$ and size $m_H$. Then $H$ is triangle-free, so $n_H \geq f(k)$. By Lemma~\ref{lemma: density} and the fact that $m \leq n + \frac{5.5(k-2)^2}{4}$, we have
$$
\frac{(k-3)\,n_H}{2} \leq m-n \leq \frac{5.5(k-2)^2}{4}.
$$
We have four cases.
\\

\noindent \textbf{Case 1.} $k = 3$. By Lemma~\ref{lemma: density}, we have $m \geq n$. We also know that $m \leq (k-2) \sqrt{5.5n}$, so
$$
n \leq m \leq \sqrt{5.5n},
$$
so $n \leq 5$. It is easy to see that any $3$-critical graph is isomorphic to an odd cycle, so $H$ is an odd cycle, and since $G$ is triangle-free, the length of this cycle is at least $5$. Thus $5 \leq n_H \leq n \leq 5$ and therefore $H \cong C_5$. Now, $m \leq \sqrt{5.5n} = \sqrt{27.5} < 6$ gives $m = 5 = m_H$. Therefore, $G \cong C_5$, a contradiction.
\\

\noindent \textbf{Case 2.} $k = 4$. Then $\frac{(k-3)\,n_H}{2} \leq \frac{5.5(k-2)^2}{4}$ gives $n_H \leq 11$, so $n_H = 11$ because $n_H \geq f(4) = 11$~\cite{chvatal1974minimality}. Hence Lemma~\ref{lemma: density} gives $m \geq n + \frac{11}{2}$, and so $m \geq n+6$. We also have $m \leq (k-2)\sqrt{5.5n}$. Therefore $(n+6)^2 \leq 22n$, which is impossible.
\\

\noindent \textbf{Case 3.} $k = 5$. Then $\frac{(k-3)\,n_H}{2} \leq \frac{5.5(k-2)^2}{4}$ gives $n_H \leq 12.375 < 22 = f(5)$, see~\cite{jensen1995small}, a contradiction. 
\\

\noindent \textbf{Case 4.} $k \geq 6$. We know that $\frac{(k-3)\,n_H}{2} \leq \frac{5.5(k-2)^2}{4}$, and by Lemma~\ref{lemma: f(k)}, we have $n_H \geq \frac{k^2-k+4}{2}$. So
$$
\frac{(k-3)\,\frac{k^2-k+4}{2}}{2} \leq \frac{5.5(k-2)^2}{4}.
$$
Equivalently,
$$
2k^3-19k^2+58k-68 \leq 0.
$$
The derivative of the left-hand side is $6k^2-38k+58$, whose largest root is $\frac{19+\sqrt{13}}{6} < 4$, so the left-hand side is increasing for $k \geq 6$ and thus is at least $28$ for $k \geq 6$, a contradiction and the proof is complete.
\end{proof}

Theorem~\ref{theorem: triangle-free main} immediately settles Problem~\ref{problem: triangle-free}.

\begin{corollary} \label{corollary: problem}
If $G$ is a connected triangle-free $(m,n)$-graph, then
$$\chi(G) \leq \left\lceil \frac{m}{2\sqrt{n}} \right\rceil + 1.
$$
\end{corollary}

\begin{remark} \label{remark: sharpness}
Both bounds in Theorems~\ref{theorem: triangle-free main} and~\ref{corollary: problem} are attained.
\begin{itemize}
  \item[(i)] Equality holds in Corollary~\ref{corollary: problem} for $C_5$.
  \item[(ii)] Equality holds in Theorem~\ref{theorem: triangle-free main} for the Gr\"{o}tzsch graph $G_{11}$ of Figure~\ref{figure: grotzsch}, for which $n = 11$, $m = 20$ and $\left\lceil \frac{20}{\sqrt{60.5}} \right\rceil + 1 = 4 = \chi(G_{11})$, as well as for every odd cycle $C_\ell$ with $7 \leq \ell \leq 21$.
\end{itemize}
\end{remark}

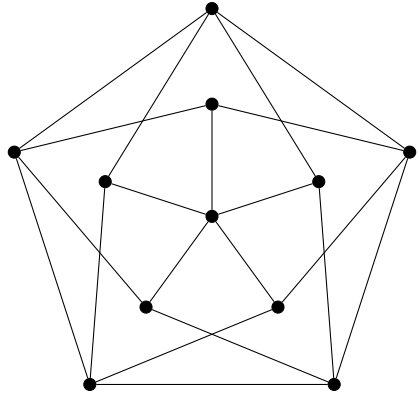
\begin{figure}[ht]
\centering
\begin{tikzpicture}[scale=1.1, every node/.style={circle, draw, fill=black, inner sep=1.6pt}]
  \foreach \i in {1,...,5} {
    \node (u\i) at ({90+72*(\i-1)}:2.5) {};
    \node (v\i) at ({90+72*(\i-1)}:1.35) {};
  }
  \node (w) at (0,0) {};
  \draw (u1)--(u2)--(u3)--(u4)--(u5)--(u1);
  \draw (v1)--(u2); \draw (v1)--(u5);
  \draw (v2)--(u3); \draw (v2)--(u1);
  \draw (v3)--(u4); \draw (v3)--(u2);
  \draw (v4)--(u5); \draw (v4)--(u3);
  \draw (v5)--(u1); \draw (v5)--(u4);
  \foreach \i in {1,...,5} { \draw (w)--(v\i); }
\end{tikzpicture}
\caption{The Gr\"{o}tzsch graph, the unique triangle-free $4$-chromatic graph of order $11$~\cite{chvatal1974minimality}}
\label{figure: grotzsch}
\end{figure}

The following remark locates the optimal constant in Theorem~\ref{theorem: triangle-free main} within a short interval.

\begin{remark} \label{remark: constant}
The constant $5.5$ in Theorem~\ref{theorem: triangle-free main} cannot be replaced by any $c \geq \sqrt{6}$. Indeed, let $G$ be the graph obtained from $C_5$ by attaching one pendant vertex. Then $G$ is connected and triangle-free with $n = m = 6$ and $\chi(G) = 3$, while $\left\lceil \frac{6}{c\sqrt{6}} \right\rceil + 1 = 2$ whenever $c \geq \sqrt{6}$.
\end{remark}

\bigskip

\bibliographystyle{plainurl}
\bibliography{refs.bib}

\end{document}